\documentclass{amsart}
\usepackage{amsmath, amsthm, amsfonts, amsaddr, enumerate, fullpage}
\usepackage[all]{xy}
\usepackage[inline]{enumitem}
\newtheorem{thm}{Theorem}

\newtheorem{cor}[thm]{Corollary}

\theoremstyle{definition}

\theoremstyle{remark}

\numberwithin{subcase}{case}

\begin{document}
\title{A combinatorial proof on partition function parity}
\author{Daniel C. McDonald}
\address{Department of Mathematics, University of Illinois, 
Urbana, IL, USA}
\email{dmcdona4@illinois.edu}
\date{}
\maketitle 
A \emph{partition} $\lambda$ of a positive integer $n$ is a nonincreasing list of positive integer \emph{parts} $\lambda_1,\ldots ,\lambda_k$ that sum to $n$.  The \emph{partition function} $p(n)$ counts the partitions of $n$.  

The number-theoretic properties of $p(n)$ have been studied extensively.  For example, Kolberg \cite{K} proved in 1959 and Newman \cite{N} proved independently in 1962 that $p(n)$ takes each value of parity infinitely often, with Fabrykowski and Subbarao \cite{FS} and Robbins \cite{R} giving new proofs of this result in 1990 and 2004, respectively.  Subbarao \cite{S} strengthened the result in 1966 by proving that $p(2n+1)$ takes each value of parity infinitely often, though he was unable to prove the analogous result for $p(2n)$; this was later proved by Kolberg in private correspondence to Subbarao.  Subbarao conjectured that $p(tn+r)$ takes each value of parity infinitely often for every pair $r$ and $t$ of integers satisfying $0\leq r<t$.  Over the years several authors confirmed this conjecture for various values of $t$, including the case $t=16$ by Hirschhorn and Subbarao \cite{HS} in 1988.  In 1995, Ono \cite{O} proved that $p(tn+r)$ is either always even or takes each value of parity infinitely often.  All of these proofs rely to some extent on manipulating generating functions.

We give a new self-contained proof that both $p(2n)$ and $p(2n+1)$ take each value of parity infinitely often.  We show these results follow from a more general theorem concerning the enumeration of certain partitions of integers along arithmetic progressions, whose proof relies on a series of bijections rather than generating functions.  For positive integers $a$ and $b$, let $D_{a,b}(n)$ be the set of partitions of $n$ into distinct parts each congruent to $b$ modulo $a$.
\begin{thm}\label{main}

Let $a,b,c,d$ satisfy $a\geq b\geq 1$, $c\geq 0$, and $d\geq 2$.  For $A=\{n:n\equiv bc\mod a\}$, there exist integers $r$ and $s$ satisfying $0\leq r<s<d$ such that $|D_{a,b}(n)|\equiv r\mod d$ for infinitely many $n\in A$ and $|D_{a,b}(n)|\equiv s\mod d$ for infinitely many $n\in A$.
\end{thm}
\begin{proof}
To show at least two congruence classes modulo $d$ are hit by $|D_{a,b}(n)|$ for infinitely many $n\in A$, it suffices to show that for every $m$ there exists $n\in A$ satisfying $n\geq m$ and $|D_{a,b}(n-a)|\not\equiv |D_{a,b}(n)|\mod d$.  Set $D^1_{a,b}(n)=D_{a,b}(n)$, and for $j>1$ set $D^{j}_{a,b}(n)=\{\lambda\in D_{a,b}(n):\lambda_1-\lambda_2=\ldots=\lambda_{j-1}-\lambda_{j}=a\}$ (partitions in $D^{j}_{a,b}(n)$ have $j$ parts or more). 

Note that $D^{j+1}_{a,b}(n)\subseteq D^{j}_{a,b}(n)$ for $j\geq 1$.  Since all parts of partitions in $D^{j}_{a,b}(n)$ lie in the same congruence class modulo $a$, a partition $\lambda\in D^{j}_{a,b}$ fails to be in $D^{j+1}_{a,b}(n)$ when $\lambda_{j+1}$ does not exist or $\lambda_{j}-\lambda_{j+1}=ta$ with $t>1$.  If $D^{j+1}_{a,b}(n)\neq\emptyset$, then $n\geq\sum_{i=0}^{j}(ai+b)$, so every partition $\lambda\in D^{j}_{a,b}(n)$ satisfies $\lambda_{j}\geq a+b$ since otherwise $n=\sum_{i=0}^{j-1}(ai+b)$.  

For $j\geq 1$ and $D^{j+1}_{a,b}(n)\neq\emptyset$, there exists a bijection $\phi^{j}_n:(D^{j}_{a,b}(n)-D^{j+1}_{a,b}(n))\rightarrow D^{j}_{a,b}(n-aj)$ defined by $(\phi^{j}_n(\lambda))_i=\lambda_i-a$ for $1\leq i\leq j$ and $(\phi^{j}_n(\lambda))_i=\lambda_i$ for $i>j$.  Thus $|D^{j}_{a,b}(n-aj)|\not\equiv |D^{j}_{a,b}(n)|\mod d$ if $|D^{j+1}_{a,b}(n)|\not\equiv 0\mod d$.

Let $k=am+c$ and $n_1=\sum_{i=0}^{k-1}(ai+b)$, so $n_{1}\equiv bc\mod a$.  Consider the partition $\lambda$ of $n_1$ with $k$ parts given by $\lambda_i=a(k-i)+b$; hence $|D^k_{a,b}(n_1)|=1\not\equiv 0\mod d$ since clearly $\lambda$ is the only partition in $D^k_{a,b}(n_1)$.  This yields $|D^{k-1}_{a,b}(n_{1}-a(k-1))|\not\equiv |D^{k-1}_{a,b}(n_{1})|\mod d$, so we can pick $n_{2}\in\{n_{1}-a(k-1),n_{1}\}$ to satisfy $|D^{k-1}_{a,b}(n_{2})|\not\equiv 0\mod d$.  Similarly, $|D^{k-2}_{a,b}(n_{2}-a(k-2))|\not\equiv |D^{k-2}_{a,b}(n_{2})|\mod d$, so we can pick $n_{3}\in\{n_{2}-a(k-2),n_{2}\}$ to satisfy $|D^{k-2}_{a,b}(n_{3})|\not\equiv 0\mod d$.  Iterate this process to compute the sequence $n_{1},n_{2},\ldots,n_{k-1}$.  

Putting everything together, we have $n_i\equiv n_{1}\equiv bc\mod a$ and $|D^{k-i}_{a,b}(n_i-a(k-i))|\not\equiv |D^{k-i}_{a,b}(n_i)|\mod d$ for $i<k$, with $n_{k-1}\geq n_{1}-\sum_{i=2}^{k-1}ai>\sum_{i=0}^{k-1}(ai+b-ai)=kb\geq m$.  Since $D_{a,b}(n)=D^1_{a,b}(n)$, setting $n=n_{k-1}$ completes the proof.
\end{proof}
The \emph{Ferrers diagram} of a partition $\lambda$ is a pattern of upper left-justified dots, with $\lambda _i$ dots in the $i$th row from the top.  The \emph{conjugate partition} of $\lambda$ is the partition whose Ferrers diagram has $\lambda _i$ dots in the $i$th column from the left.
\begin{cor}
Both $p(2n)$ and $p(2n+1)$ take each value of parity infinitely often.

\end{cor}
\begin{proof}
Partition conjugation is an involution on the set of partitions of $n$ that fixes only the partitions whose Ferrers diagrams are symmetric about the diagonal from the upper left to lower right.  Thus $p(n)$ has the same parity as the number of self-conjugate partitions of $n$, and the set of such partitions is in one-to-one correspondence with the set of partitions of $n$ into distinct odd parts through the bijection that unfolds the Ferrers diagram of any self-conjugate partition about its axis of symmetry.  Thus $p(n)\equiv |D_{2,1}(n)|\mod 2$.  Applying Theorem \ref{main} twice with $(a,b,c,d)=(2,1,0,2)$ and $(a,b,c,d)=(2,1,1,2)$ yields both claims.
\end{proof}

\end{document}